\documentclass[11pt]{amsart}

\usepackage{DKstyle}
\usepackage{bm}
\usepackage{float}
\usepackage[toc,page]{appendix}

\newcommand{\vertiii}[1]{{\left\vert\kern-0.25ex\left\vert\kern-0.25ex\left\vert #1
    \right\vert\kern-0.25ex\right\vert\kern-0.25ex\right\vert}}
 \makeatletter
\newcommand*{\rom}[1]{\expandafter\@slowromancap\romannumeral #1@}
\makeatother
\theoremstyle{plain}

\usepackage{caption}
\usepackage[labelfont=rm]{subcaption}

\usepackage[pagebackref=false, colorlinks]{hyperref}

\hypersetup{pdffitwindow=true,linkcolor=blue,citecolor=blue,urlcolor=blue,menucolor=blue}

\usepackage{comment}


\date{\today}%
\dedicatory{}%
\commby{}%

\title{Counting rational points on the sphere with bounded denominator}
\author{Dubi Kelmer}
\address{Department of Mathematics, Boston College, Chestnut Hill MA 02467-3806, USA}
\email{kelmer@bc.edu}

\thanks{This work was partially supported by NSF CAREER grant DMS-1651563.}
\begin{document}
\begin{abstract}
We give an optimal bound for the remainder when counting the number of rational points on the $n$-dimensional sphere with bounded denominator for any $n\geq 2$.
\end{abstract}
\maketitle


\section{Introduction}
For any $n\in \N$ let $S^n=\{\bm{x}\in \R^{n+1}: \|\bm{x}\|=1\}$, denote the unit sphere in $\R^{n+1}$ where $\|\bm{x}\|=\sqrt{x_1^2+\ldots+x_n^2}$ is the Euclidian norm. 
We are interested in the counting function 
$$N(S^n;T)=\#\{\frac{\bm{p}}{q}\in S^n: q\leq T\}$$
counting the number of rational points on the sphere with bounded denominator, where we assume $\frac{\bm{p}}{q}$ is written in lowest terms.
Similarly, for any subset $\Omega\subset S^n$ we let 
$$N(\Omega; T)=\#\{\frac{\bm{p}}{q}\in \Omega: q\leq T\}.$$
Asymptotically  $N(S^n;T)\sim c_n T^n$ for some $c_n>0$, and more generally $N(\Omega; T)\sim c_n\mu_n(\Omega)T^n$ as $T\to\infty$ where $\mu_n$ is the rotation invariant probability measure on $S^n$. There are many methods for obtaining such estimates, including the circle method \cite{HeathBrown1996, Getz2018}, harmonic analysis on groups 
\cite{DukeRudnickSarnak1993, GorodnikNevo2012, BenoistOh2012}, analysis of $L$-functions of modular forms \cite{Duke2003,BurrinGrobner2024}, and the light cone Siegel transform \cite{KelmerYu2023a, KelmerYu2023b}; each attacking the problem from a different viewpoint and allows for different generalizations. 

Recently, in \cite{BurrinGrobner2024}, using the well known interpretation of sum of squares function as a Fourier coefficient of a modular form, building on the ideas in \cite{Duke2003} and relating these counting functions to contour integrals of certain $L$-functions attached to modular forms, the authors obtain effective bounds for the remainder: showing that 
$$N(S^n;T)=c_nT^n(1+O_\epsilon(T^{-\frac{n}{n+1}+\epsilon})),$$
improving on the results of \cite{KelmerYu2023a}.  Moreover, they showed that, for even $n$ and assuming GRH, one can sharpen the remainder term further to $O_\epsilon(T^{-\frac{n+1}{n+2}+\epsilon})$. They also showed that for even $n$ and  $\Omega$ with piecewise smooth boundary
$$N(\Omega;T)=c_n\mu_n(\Omega)T^n(1+O_{\Omega}(T^{-\frac{n}{2n+2}+\epsilon})),$$ 
thus improving on exponent previously obtained in \cite{KelmerYu2023b}. Here and below we use the notation $f(T)=O(g(T))$ and also $f(T)\ll g(T)$, if there is a constant $c$ such that $|f(T)|\leq cg(T)$ for all $T$; we use subscripts to indicate the dependence of the constant on additional parameters. 

The goal of this note is to show that it is possible to replace the argument involving contour integration of $L$-functions of modular forms with a different argument just using bounds on Fourier coefficients of cusp forms together with  an elementary estimate involving sums of twisted divisor functions. Moreover, for estimating $N(S^m;T)$ for $n\geq 2$, this approach gives a sharper (and essentially optimal) bound for the remainder\footnote{I thank Zeev Rudnick for pointing out that such an approach should be possible for this problem.}. Explicitly we show the following.
\begin{thm}\label{t:main}
For any $n\geq 2$ there is a constant $c_n>0$  and $\alpha_n\in \{0,1,3\}$ such that 
$$N(S^n;T)=c_nT^n(1+O(\frac{(\log T)^{\alpha_n}}{T})),$$
where $\alpha_n=0$ for odd $n\geq 3$, $\alpha_n=1$ for even $n\geq 4$ , $\alpha_3=1$ and $\alpha_2=3$. \end{thm}

\begin{rem}
Noting that for $T\in \N$ we have that $N(S^n,T)=N(S^n, T+\frac12)$ we see that, up to the factors of $\log(T)$, this bound is sharp.
\end{rem}

\begin{rem}
By following the calculation carefully it is possible to express $c_n$ explicitly as a ratio of special values of Dirichlet $L$-functions. However, for the sake of brevity, and since the value of $c_n$ was already calculated explicitly in \cite{KelmerYu2023a}, we will not do this here.
\end{rem}

\begin{rem}
The same method can also be used to estimate  $N(\Omega;T)$ for $\Omega\subset S^n$  and give a power saving error term for the remainder (see \secref{s:NOmega}). However, in this case the result is not as good as the one obtained in \cite{BurrinGrobner2024}, so we will not include the full details of this calculation. 
\end{rem}

\begin{rem}\label{r:SQ}
In  \cite{BurrinGrobner2024}, the authors considered rational point on more general ellipsoids 
$$S_Q=\{x\in \R^{n+1}: Q(x)=1\},$$ 
with $Q(x)=\frac{1}{2}x^tAx$ with $A$ an integral symmetric positive definite matrix such that there is $N=4D$ with $D$ an odd square free integer, such that   $NA^{-1}$is integral with even diagonal entries. We note that the methods of this paper also generalize to this setting when $n\geq 3$ (see \secref{s:SQ}).
\end{rem}

\section{Preliminaries}
We start with some basic notation and definitions from the theory of modular forms. 
Let $\Gamma=\SL_2(\Z)$ acting on the the upper half plane
$\bH=\{z=x+iy:y>0\}$ by linear fractional transformation. For any $N\in \N$ we denote by 
$$\Gamma_0(N)=\{\begin{pmatrix} a& b\\ c&d\end{pmatrix}: c\equiv 0\pmod{N}\}.$$

\subsection{Theta functions}
For any $n\in \N$ the standard Theta function is defined by  
$$\Theta_n(z)=\sum_{\bm{v}\in \Z^n} e^{2\pi i \|\bm{v}\|^2 z}=\sum_{\ell =1}^\infty r_n(\ell)e^{2\pi i \ell z}$$

The function $\Theta_n(z)$ is holomorphic  on the upper half plane and satisfies that for any $\gamma=\begin{pmatrix} a& b\\ c&d\end{pmatrix}\in\Gamma_0(4)$ 
$$\Theta_n(\gamma z)=\nu^n_\theta(\gamma)(cz+d)^{n/2}\Theta_n(z),$$
where 
$\nu_\theta(\gamma)=\bar\epsilon_d (\frac{c}{d})$ is the multiplier system for the Theta function with 
$$\epsilon_d=\left\lbrace \begin{array}{cc}1 & d\equiv 1\pmod{4}\\
i & d\equiv -1\pmod{4}\end{array}\right.,$$
and $(\frac{c}{d})$ is Shimura's extention of the Jacobi symbol, defined for negative odd $d\in \Z$ by 
$(\frac{c}{d})=\frac{c}{|c|}(\frac{c}{-d})$
for $c\neq 0$ and $(\frac{0}{d})=0$ unless $d=\pm1$ in which case $(\frac{0}{d})=1$ (see \cite[Corollary 10.7]{Iwaniec1997book}) . 

We note for future reference that for any integer  $m\neq 0$ the function $\chi(d)=(\frac{m}{d})$, is a Dirichlet character modulo $4|m|$, and we denote by $\omega_m$ the unique primitive Dirichlet character satisfying that 
\begin{equation}\label{e:omegam}
\omega_m(d)=\left(\frac{m}{d}\right),\quad \forall d\in \Z,\; (d,4m)=1.\end{equation}
Note that $\omega_m$ depends only on the squarefree part of $m$.

\subsection{The space of modular forms of integral and half integral weight}
For any $k\in \N$ and $\chi$ a Dirichlet character modulo $N$ the space $M_k(N,\chi)$ of modular forms of weight $k$ and Nebentypus $\chi$ for $\Gamma_0(N)$  is the space of all holomorphic function on $\bH\cup \Q\cup \infty$ satisfying that 
$f(\gamma z)=\chi(\gamma)(cz+d)^kf(z)$ for all $\gamma\in \Gamma_0(N)$, where the character $\chi$ is extended to $\Gamma_0(N)$ via $\chi(\gamma)=\chi(d)$.

For any $k\in \tfrac{1}{2}+\N$, and $N\in \N$ with $4|N$ and $\chi$ a Dirichlet character modulo $N$, the space $M_k(N,\chi)$ of modular forms of weight $k$ for $\Gamma_0(N)$  is the space of all holomorphic function on $\bH\cup \Q\cup \infty$ satisfying that 
$f(\gamma z)=\nu_\theta(\gamma)^{2k}\chi(\gamma)(cz+d)^kf(z)$ for all $\gamma\in \Gamma_0(N)$. 

In both cases the space $M_k(N,\chi)$ is finite dimensional and can be decomposed as a direct sum
$$M_k(N,\chi)=\cE_k(N,\chi)\oplus \cS_k(N,\chi)$$
with $\cS_k(N,\chi)$ the space of cusp forms vanishing at all cusps, and $\cE_k(N,\chi)$ its orthogonal complement with respect to the Petersson Inner Product. The space $\cE_k(N,\chi)$ is spanned by Eisenstein series that can be explicitly calculated.

In particular, we see that the Theta function $\Theta_n\in M_{\frac{n}{2}}(4,\chi)$, where $\chi$ is the principal character if $n$ is odd or if $n\equiv 0\pmod{4}$ and $\chi=\omega_{-1}$ is the quadratic character to modulus $4$, if $n\equiv 2\pmod{4}$.  We can thus decompose
$$\Theta_n=E+f$$
with $E\in \cE_{\frac{n}{2}}(4,\chi)$ a sum of Eisenstein series and $f\in \cS_{\frac{n}{2}}(4,\chi)$ a cusp form.

\subsection{Fourier coefficients of cusp forms}\label{s:Eisenstein}
For any $f\in \cS_k(N,\chi)$ we can expand 
$$f(z)=\sum_{n=1}^\infty a_f(n)e^{2\pi i n z},$$
with $a_f(n)$ the $n$'th Fourier coefficient of $f$ at the cusp at $\infty$. 
For integral $k\geq 2$ the Ramanujan-Petersson bound, proved by Deligne \cite{Deligne1974}, implies that   
\begin{equation}\label{e:Deligne}
|a_f(n)|\ll_f \sigma_0(n)n^{\frac{k-1}{2}},
\end{equation}
with $\sigma_0(n)=\sum_{d|n}1$ is the number of divisors of $n$ (and the implied constant is absolute if $f$ is a primitive Hecke newform).

For $k$ a half integer, while the Ramanujan-Petersson bound is not known in general we have Hecke's bound 
$$|a_f(n)|\ll n^{\frac{k}{2}},$$
which holds in general. Moreover, assuming $k\geq \frac{5}{2}$, when evaluating at perfect squares the Shimura correspondence \cite{Shimura1973} implies that 
$$a_f(n^2)=\sum_{d|n}\mu(d)d^{k-3/2} a_{\tilde{f}}(\tfrac{n}{d}),$$
where $\tilde{f}\in S_{2k-1}(\tfrac{N}{2},\chi^2)$ is the Shimura lift of $f$. Here $\mu(d)$ denotes the M\"obius function, which is defined as the unique multiplicative function supported on square free integers and satisfying that $\mu(p)=(-1)$ for any prime $p$. Applying Deligne's bound \eqref{e:Deligne} for $a_{\tilde{f}}$ we get the following bound for Fourier coefficients of $f\in S_k(N,\chi)$ at perfect squares:
\begin{equation}\label{e:Shimura}
|a_f(n^2)|\ll_f \sum_{d|n}|\mu(d)| d^{k-3/2} \sigma_0(\tfrac{n}{d}) (\frac{n}{d})^{k-1}. 
\end{equation}
\subsection{Fourier coefficients of Eisenstein series}
For integral weight $k\geq 3$ and a character $\chi$ modulo $N$,  a basis for the space of $\cE_k(N,\chi)$ composed of Eisenstein series $G_k(\chi_1,\chi_2,z)$,  parametrized by pairs of Dirichlet characters $\chi_1$ modulo $N_1$ and $\chi_2$ modulo $N_2$ with $N=N_1N_2$, $\chi=\chi_1\chi_2$ and $\chi_1$ primitive, was given in \cite[Theorem 8.5.17]{CohenStromberg2017}.  Moreover,  by \cite[Theorem 8.5.5]{CohenStromberg2017}  the $n$th Fourier coefficient of $G_k(\chi_1,\chi_2,\tau)$ is  given by some constant multiple of the  twisted divisor sums 
$\sigma_{k-1}(\chi_1,\chi_2,n)$
where
\begin{equation}\label{e:divisor1}
\sigma_s(\chi_1,\chi_2,n)=\sum_{d|n} \chi_1(d)\chi_2(\tfrac{n}{d})d^s.
\end{equation}

For half integral weight $k\in \tfrac{1}{2}+\N$ with $k\geq 5/2$ for any square free odd $D$ and $m\in \Z$ with $m|D$, a basis for $\cE_k(4D,\omega_m)$ comprised of Eisenstein series with an explicit formula for their Fourier coefficients was given in \cite[Theorem 7.1]{WangPei2012}. The Fourier coefficient at $n\in \N$  is given by a product of a divisor sum and a special value of the Dirichlet $L$-function $L(k-\frac{1}{2},\omega_{\pm n})$. In particular, when  $n$ is a perfect square these $L$-values are independent of $n$ and it is possible to express the Fourier coefficients as a modified divisor sums. 

To simplify the notation we will describe their results for the special case of $D=1$  (which is all we will need for our result). 
Let $s_k=(-1)^{k-1/2}$ and for any odd $m\in \N$ consider the modified divisor sum
\begin{equation}\label{e:divisor2}
\beta_k(m)=m^{2k-2}\sum_{ab| m}\mu(a)\left(\frac{s_k}{a}\right) a^{1/2-k}b^{2-2k}.
\end{equation}
Then the space $\cE_k(4)$ has a basis of Eisenstein series, each has a Fourier expansion of the form
$$G_k(z)=\sum_{n=0}^\infty \lambda_k(n)e^{2\pi i nz},$$
where $\lambda_k(0)$ can be  either $0$ or $1$  and for $n=2^\nu m$ with $m$ odd we have that 
\begin{equation}\label{e: lambdak}
\lambda_k(n^2)=c_k 2^{2\nu(k-1)} A_k(\nu)\beta_k(m),
\end{equation}
with $c_k$ a constant independent of $n$ and either $A_k(\nu)=1$ for all $\nu$  or $A_k(\nu)$ is a rational function of $2^\nu$  satisfying $|A_k(\nu)|\leq 1$.
  
 
 \subsection{Dirichlet L-values}
For any Dirichlet character $\chi$ modulo $N$, and $s>1$ we let 
$$L(\chi,s)=\sum_n \frac{\chi(n)}{n^s}=\prod_p(1-\chi(p)p^{-s})^{-1}.$$
In particular the zeta function $\zeta(s)=L(1,s)$ for $\chi=1$ trivial.
While these $L$-function have an analytic continuation for $s\in \C$ we will only need to consider real values $s>1$ where the series absolutely converges.
For another integer $M\in \N$ we denote by 
$$L_M(\chi,s)=\sum_{(n,M)=1} \frac{\chi(n)}{n^s}=\prod_{p\not | M}(1-\chi(p)p^{-s})^{-1}.$$
We denote by $\mu(n)$ the M\"obius function.
Using the Euler product expansion it is not hard to see that 
$$\frac{1}{L(\chi,s)}=\sum_n \frac{\mu(n)\chi(n)}{n^s}=\prod_p(1-\chi(p)p^{-s})$$
and 
$$\frac{L(\chi,s)}{L(\chi^2,2s)}=\sum_n \frac{|\mu(n)|\chi(n)}{n^s}=\prod_p(1+\chi(p)p^{-s}).$$

\section{Proof of main result }
\subsection{Reduction to  Fourier coefficients of $\Theta_{n+1}$}
For a modular form $f\in M_k(N,\chi)$ with Fourier coefficient $a_f(m)$ consider the counting function 
$$N_k(f;T)=\sum_{ q\leq T} a_f(q^2).$$
 Recall that  
 $$r_{n+1}(m)=\#\{\bm{v} \in \Z^{n+1}: \|\bm{v}\|^2=m\},$$ is the $m$'th Fourier coefficient of $\Theta_{n+1}(z)$ which is a modular form in $M_k(4,\chi)$ with $k=\frac{n+1}{2}$ and $\chi$ a suitable real Dirichlet character. We thus have the counting function 
 $$N_k(\Theta_{n+1};T)=\sum_{q\leq T} r_{n+1}(q^2).$$ 
We show the following
\begin{lem}\label{l:N2N*}
For any $\alpha\geq 0$ and any $n\geq 3$ 
$$N(S^n;T)=c_nT^n+O(T^{n-1}\log^\alpha(T)) \quad \Leftrightarrow \quad N_k(\Theta_{n+1};T)=c^*_nT^n+O(T^{n-1}\log^\alpha(T)),$$ 
with  $c_n^*=\zeta(n)c_n$. For $n=2$ the estimate $N_k(\Theta_{3};T)=c^*_2T^2+O(T\log^\alpha(T))$ implies that 
$N(S^2;T)=c_2T^2+O(T\log^{\alpha+1}(T))$ (and similarly in the other direction).

\end{lem}
\begin{proof}
We can write 
\begin{eqnarray*}N_k(\Theta_{n+1};T)&=&\#\{(\bm{p},q)\in\Z^{n+2}: \|\bm{p}\|^2=q^2,\; q\leq T\}\\
&=& \sum_{d=1}^{[T]}\#\{(\bm{p},q)\in\Z^{n+2}: \|\bm{p}\|^2=q^2,\; q\leq T, \gcd(\bm{p},q)=d\}\\
&=& \sum_{d=1}^{[T]}\#\{(\bm{p},q)\in\Z^{n+2}: \|\bm{p}\|^2=q^2,\; q\leq \frac{T}{d}, \gcd(\bm{p},q)=1\}\\
&=& \sum_{d=1}^{[T]} N(S^n; \frac{T}{d})
\end{eqnarray*}
Conversely, using M\"obius inversion.
\begin{eqnarray*}
\sum_{d=1}^{[T]} \mu(d)N_k(\Theta_{n+1};\frac{T}{d})&=& \sum_{d=1}^{[T]}\mu(d)\sum_{e=1}^{[\frac{T}{d}]} N(S^n;\frac{T}{de})\\
&=& \sum_{m=1}^{[T]}\sum_{d|m}\mu(d)N(S^n;\frac{T}{m})=N(S^n;T).
\end{eqnarray*}
The result now follows after noting that for any $n\geq 2$ we have  $\sum_{d=1}^\infty \frac{\mu(d)}{d^n}=\frac{1}{\zeta(n)}$ and that $\sum_{d=1}^\infty \frac{1}{d^{n-1}}$ converges for $n\geq 3$ and $\sum_{d=1}^T \frac{1}{d}\ll \log(T)$. 
\end{proof}

Next, to evaluate $N_k(\Theta_{n+1};T)$ decompose 
$$\Theta_{n+1}(z)=E(z)+f(z)$$
with $E\in \cE_{\frac{n+1}{2}}(4,\chi)$ a sum of Eisenstein series, and $f\in \cS_{\frac{n+1}{2}}(4,\chi)$ a cusp form.
Hence the Fourier coefficients $r_{n+1}(m)$ can be written as a sum of a Fourier coefficients of an Eisenstein series and  a Fourier coefficient of a cusp form, consequently the counting function
$$N_k(\Theta_{n+1},T)=N_k(E;T)+N_k(f;T).$$
Next we will bound the contribution of the cusp forms and give an asymptotic estimates to that for Eisenstein series.

\subsection{Bounding contribution of cusp forms}
The following estimates will show that the contribution of cusp forms  is negligible.
 \begin{prop}\label{p:Ncusp}
 For any $k\geq 2$ integral or half integral, for any cusp form  $f\in S_k(N,\chi)$ 
$$N_k(f;T)\ll_k    \|f\|_2T^k\log^\alpha(T),$$
with $\alpha=2$ when $k\in \N$ and $\alpha=1$ when $k\in \N+\frac12$.
\end{prop}
\begin{proof}
For integral $k\geq 2$ using Deligne's bound \eqref{e:Deligne}
we can bound 
\begin{eqnarray*} 
N_k(f;T) \leq  \sum_{q\leq T}|a_f(q^2)| \ll_f \sum_{q=1}^\infty \sigma_0(q^2)q^{k-1}\Psi_T(q)\\
\end{eqnarray*}
where $\Psi_T(q)=1$ for $q\leq T$ and zero otherwise. Expanding $\sigma_0(q^2)$ and noting that any divisor $d|q^2$ can be written in a unique way as $d=ab^2$ with $a$ square free and $q\equiv 0\pmod{ab}$ , changing the order of summation we get 
\begin{eqnarray*} 
N_k(f;T)& \ll_f &\sum_{q} (\sum_{d|q^2} 1)q^{k-1}\Psi_T(q)\\
&=& \sum_{a=1}^\infty |\mu(a)|\sum_{b=1}^\infty\mathop{\sum_{q=1}^\infty}_{ab|q}q^{k-1}\Psi_T(q)\\
&=& \sum_{a\leq T}|\mu(a)|a^{k-1}\sum_{b\leq \frac{T}{a}} b^{k-1}\sum_{c\leq \frac{T}{ab}} c^{k-1}\\
&\leq & T^k \sum_{a\leq T}\frac{|\mu(a)|}{a} \sum_{b\leq \frac{T}{a}} \frac{1}{b}\leq  T^k\log^2(T).
\end{eqnarray*}

Next, for half integral $k\geq \frac{5}{2}$ using the bound \eqref{e:Shimura} for $|a_f(q^2)|$ we can bound  
\begin{eqnarray*} 
N_k(f;T)&\leq &\sum_{q\leq T}|a_f(q^2)|\\
&\ll_f &\sum_{q\leq T} \sum_{a|q}|\mu(a)| a^{k-3/2} \sigma_0(\tfrac{q}{a}) (\frac{q}{a})^{k-1}\\
&=&  \sum_{a\leq T}|\mu(a)| a^{k-3/2} \sum_{b\leq \frac{T}{a}}(\sum_{d|b} 1)b^{k-1}\\
&=& \sum_{a\leq T}|\mu(a)a^{k-3/2}\sum_{d\leq \frac{T}{a}}d^{k-1}(\sum_{b\leq\frac{T}{ad}}b^{k-1})\\
&\leq & \sum_{a\leq T} a^{k-3/2}\sum_{d\leq \frac{T}{a}}d^{k-1}(\frac{T}{ad})^{k}\\
&=&T^k  \sum_{a\leq T}a^{-3/2}\sum_{d\leq \frac{T}{a}}\frac{1}{d} \ll  T^k\log(T)
\end{eqnarray*}

Finally to make the dependence on $f$ explicit fix an orthonormal basis $f_1,\ldots,f_{d_k}$ for the space of cusp forms.
There is a constant, $C_k$, depending only on $k$, such that $N_k(f_j,T)\leq C_kT^k\log(T)$ for all $1\leq j\leq d_k$. Now for any $f\in \cS_k(N,\chi)$ we can expand
$$f=\sum_j \langle f,f_j\rangle f_j.$$
Hence 
$$N_k(f;T)=\sum_j  |\langle f,f_j\rangle N_k(f_j;T)\leq (\sum_j  |\langle f,f_j\rangle) C_kT^k\log(T) $$
Using Cauchy Schwarz we can bound $\sum_j |\langle f,f_j\rangle |\leq  \sqrt{d_k}\|f\|_2,$
concluding the proof.
\end{proof}

\subsection{Sums of twisted divisor functions}
For $k\in \N$ and  two Dirichlet characters $\chi_1,\chi_2$ modulo $N_1$ and $N_2$ we recall the twisted divisor sum $\sigma_k(\chi_1,\chi_2,m)$ given in \eqref{e:divisor1} and consider the following sum 
\begin{equation}\label{e:divisorsum}
\cS_k(\chi_1,\chi_2,T)=\sum_{q\leq T}\sigma_k(\chi_1,\chi_2,q^2).
\end{equation}
For $k\in \frac{1}{2}+\N$ a half integer recall the modified divisor sums $\beta_k(m)$ given in \eqref{e:divisor2} and consider the sum
\begin{equation}\label{e:divisorsum2}
\cB_k(T)=\mathop{\sum_{m\leq T}}_{\rm{odd}} \beta_k(m)
\end{equation}
The main goal of this section is to give an assymptotic estimate for these sums. 
We start with the following general lemma.

\begin{lem}\label{l:sumsquares}
Given completely multiplicative functions $f,g$ and the Dirichlet convolution 
$$h(m)=\sum_{d|m}f(d)g(\frac{m}{d}),$$ we have 
$$\sum_{q\leq T}h(q^2)=\sum_{a\leq T}|\mu(a)| f(a)g(a)\sum_{bc\leq \frac{T}{a}} f^2(b)g^2(c)$$
\end{lem}
\begin{proof}
First for two functions $f,g$ we have
\begin{eqnarray*}
\sum_{q\leq T}h(q^2)&=& \sum_{q\leq T}\sum_{d|q^2}f(d)g(\frac{q^2}{d})\\
&=& \sum_{d\leq T^2}\mathop{\sum_{q\leq T}}_{d|q^2}f(d)g(\frac{q^2}{d})
\end{eqnarray*}
Now write any divisor  $d=ab^2$ with $a$ square free and $b\in \N$, and note that the condition $d|q^2$ is equivalent to the condition that $ab|q$. We can thus write in the inner sum $q=abc$  to get 
\begin{eqnarray*}
\sum_{q\leq T}h(q^2)&=&\sum_{a\leq T}|\mu(a)|\sum_{bc\leq \frac{T}{a}} f(ab^2)g(ac^2)
\end{eqnarray*}
Finally using that $f,g$ are completely multiplicative we get 
\begin{eqnarray*}
\sum_{q\leq T}h(q^2)&=&\sum_{a\leq T} |\mu(a)| f(a)g(a)\sum_{bc\leq \frac{T}{a}} f(b^2)g(c^2)
\end{eqnarray*}
as claimed.
\end{proof}
\begin{prop}\label{p:SkT}
Let $k\geq 1$ and let $T\geq 1$. If $\chi_1$ is real valued we have
$$\cS_k(\chi_1,\chi_2,T)=\tfrac{\phi(N_1)}{(2k+1)N_1}\tfrac{L(\chi_2^2, 2k+1) L(\chi_1\chi_2, k+1)}{L(\chi_1^2\chi_2^2, 2k+2)}T^{2k+1}+O(T^{2k}+T^2\log(T)),$$
Otherwise, $S_k(\chi_1,\chi_2,T)=O(T^{2k}+T^2\log(T))$.
\end{prop}
\begin{proof}
Applying \lemref{l:sumsquares}  with $f(m)=\chi_1(m)m^k$ and $g(m)=\chi_2(m)$ we have 
\begin{eqnarray*}
\sum_{q\leq T} \sigma_k(\chi_1,\chi_2,q^2)&=&\sum_{a\leq T}|\mu(a)| \chi_1(a)\chi_2(a)a^k \sum_{bc\leq \frac{T}{a}} \chi_1^2(b)\chi_2^2(c)b^{2k}\\
&=& \sum_{a\leq T}| \mu(a)| \chi_1(a)\chi_2(a)a^k S(\frac{T}{a}) 
\end{eqnarray*}
where
$$S(X)= \sum_{bc\leq X} \chi_1^2(b)\chi_2^2(c)b^{2k}.$$
Using Dirichlet hyperbola method we split the sum $S(X)=S_1(X)+S_2(X)-S_3(X)$ in to terms with $c\leq \sqrt{X}$ terms with $b\leq \sqrt{X}$ and subtract the overcount of terms when both $b,c\leq \sqrt{X}$. We  estimate each one separately.
The main contribution comes from $S_1(X)$ when $\chi_1^2$ is principal, in which case
\begin{eqnarray*}
S_1(X) &=& \sum_{c\leq \sqrt{X}}\chi_2^2(c)\mathop{\sum_{b\leq \frac{X}{c}}}_{(b,N_1)=1}b^{2k}\\
&=&\frac{\phi(N_1)}{N_1} \frac{X^{2k+1}}{2k+1} \sum_{c\leq \sqrt{X}}\frac{\chi_2^2(c)}{c^{2k+1}}+O(X^{2k})\\
&=&   \frac{\phi(N_1)  L(\chi_2^2, 2k+1)}{(2k+1)N_1}X^{2k+1}+O(X^{2k}) \end{eqnarray*}
where $\phi(N_1)=\#\{1\leq b\leq N_1: (b,N_1)=1\}$ is Eulers totient function.
Note that when $\chi_1^2$ is not principal we can bound 
\begin{eqnarray*}
S_1(X) &=& \sum_{c\leq \sqrt{X}}\chi_2^2(c)\sum_{b\leq \frac{X}{c}} \chi_1^2(b) b^{2k}
\ll X^{2k}\\ \end{eqnarray*}

The contribution of the other two sums can be trivially bounded by 
\begin{eqnarray*}
S_2(X) &\leq &\sum_{b\leq \sqrt{X}} b^{2k}\sum_{c\leq \frac{X}{b}} \chi_2^2(c)\\
&=&X \sum_{b\leq \sqrt{X}} b^{2k-1}\ll X^{k+1}\\
\end{eqnarray*}
and 
\begin{eqnarray*}
S_3(X) &\leq & \sum_{b,c\leq \sqrt{X}}b^{2k}\ll X^{k+1}
\end{eqnarray*}
We thus  get that for any $k\geq 1$,
$$S(X)= \left\lbrace\begin{array}{cc}   \frac{\phi(N_1)  L(\chi_2^2, 2k+1)}{(2k+1)N_1} X^{2k+1}+O(X^{2k}) & \chi_1^2=1\\
O(X^{2k}) & \chi_1^2\neq 1
\end{array}\right.$$

Plugging this back, when $\chi_1$ is principal or quadratic
\begin{eqnarray*}
\sum_{q\leq T} \sigma_k(\chi_1,\chi_2,q^2)&=& \sum_{a\leq T}|\mu(a)| \chi_1(a)\chi_2(a)a^k S(\frac{T}{a})\\
&=&    \frac{\phi(N_1)  L(\chi_2^2, 2k+1)}{(2k+1)N_1}T^{2k+1} \sum_{a\leq T}|\mu(a)| \frac{\chi_1\chi_2(a)}{a^{k+1}} +O(T^{2k}\sum_{a\leq T} \frac{1}{a^k})\\
&=&  \frac{\phi(N_1)}{(2k+1)N_1} \frac{L(\chi_2^2, 2k+1) L(\chi_1\chi_2, k+1)}{L(\chi_1^2\chi_2^2, 2k+2)}T^{2k+1}+O(T^{2k}\log(T))
\end{eqnarray*}
where the $\log(T)$ factor is only needed when $k=1$. When $\chi_1^2$ is not principal, the main term cancels and the whole sum is bounded by $O(T^{2k})$ (resp. $O(T^2\log(T))$ if $k=1$).

\end{proof}

Next we estimate the sum of the modified divisor function.
\begin{prop}\label{p:BkT}
For any $k\in \frac{1}{2}+\N$, $k\geq 5/2$ let $s_k=(-1)^{k-1/2}$.  For any $T\geq 1$ large we have
$$\cB_k(T)= \frac{\zeta_2(2k-1)}{(4k-2)L_2(\omega_{s_k},k+\frac{1}{2})}T^{2k-1}+O(T^{2k-2}),$$
\end{prop}
\begin{proof}
Plugging in the formula for $\beta_k(m)$ and changing the order of summation we can write
\begin{eqnarray*}
\cB_k(T)&=&\mathop{ \sum_{m\leq T}}_{\rm{ odd}}m^{2k-2}\sum_{ab| m}\mu(a)\left(\frac{s_k}{a}\right) a^{1/2-k}b^{2-2k}\\
&=& \mathop{\sum'_{a\leq T} \mu(a)\left(\frac{s_k}{a}\right)a^{k-3/2}\sum'_{mb\leq \frac{T}{a}}}m^{2k-2}
\end{eqnarray*}
where $\sum'$ indicates we are only summing over odd integers.
We now use Dirichlet hyperbola method to estimate the inner sum 
$$S(X)=\sum_{mb\leq X}' m^{2k-2}=S_1(X)+S_2(X)-S_3(X)$$
where as before the first sum is over $m,b$ with $b\leq \sqrt{X}$, the second is over $m,b$ with $b\leq \sqrt{X}$ and the third the over cound where both $m\leq \sqrt{X}$ and $b\leq\sqrt{X}$. Now
\begin{eqnarray*}
S_1(X)&=&\sum'_{b\leq \sqrt{X}}\sum'_{m\leq \frac{X}{b}}m^{2k-2}
= \frac{\zeta_2(2k-1)}{4k-2} X^{2k-1}+O(X^{2k-2})\\
\end{eqnarray*}
The other two sums are trivially bounded by 
$S_2(X)\leq \sum_{m\leq \sqrt{X}}m^{2k-2} \frac{X}{m}\leq  X^k$ and  similarly $|S_3(X)|\leq X^k$.
Plugging this back in we get that 
\begin{eqnarray*}
\cB_k(T)&=&\sum'_{a\leq T} \mu(a)\left(\frac{s_k}{a}\right)a^{k-3/2}S(\frac{T}{a})\\
&=&   \frac{\zeta_2(2k-1)}{4k-2}T^{2k-1} \sum'_{a\leq T} \mu(a)\left(\frac{s_k}{a}\right)a^{-(k+1/2)}+O(T^{2k-2})\\
&=&  \frac{\zeta_2(2k-1)}{(4k-2)L_2((\frac{s_k}{\cdot}),k+\frac{1}{2})}T^{2k-1}+O(T^{2k-2})\\
\end{eqnarray*}

\end{proof}

\subsection{Proof of \thmref{t:main}}
We first consider the case of odd $n\geq 5$. In this case  $k=\frac{n+1}{2}$ is integral and $\Theta_{n+1}(z)\in M_k(4,\chi)$ with $\chi$ either trivial or $\chi=\omega_{-1}$ quadratic (depending on parity of $k=\frac{n+1}{2}$). We further note that when $n=3$ there are no cusp forms and $\Theta_3(z)\in \cE_2(4)$ which is one dimensional so $r_3(n)$ is a twisted divisor sum.
For $n\geq 5$, writing $\Theta_{n+1}(z)=E(z)+f(z)$ as a sum of Eisenstein series and cusp form and using \propref{p:Ncusp} to bound $N_k(f;T)\ll_f T^k\log^2(T)\ll T^{n-1}$  it is enough to estimate 
$N_k(E,T)$. Writing $E(z)$ as a sum of Eisenstein series having Fourier coefficients given by the twisted divisor sums, we can write $N_k(E,T)$ as a sum of finitely many sums, each a multiple of $S_{k-1}(\chi_1,\chi_2,T)$ with $\chi_1,\chi_2$ real characters. We then use \propref{p:SkT} to estimate these sums to get that
$$N_k(E,T)=c(E)T^n+O(T^{n-1}+T^2\log(T)),$$
with $c(E)$ is given by ratios of special values of Dirichlet $L$-functions.
Hence the same estimates hold for $N_k(\Theta_{n+1};T)$ and by \lemref{l:N2N*} we have  $N(S^n;T)=c_nT^n+O(T^{n-1}+T^2\log(T)$, with $c_n=\frac{c(E)}{\zeta(n)}$.

Next, for even  $n\geq 4$ we have that $k=\frac{n+1}{2}\geq \frac{5}{2}$ is a half integer. As before we can write  $\Theta_{n+1}(z)=E(z)+f(z)$ as a sum of Eisenstein series and cusp form and using \propref{p:Ncusp} to bound $N_k(f;T)\ll_f T^{n-1}$ it is enough to estimate  $N_k(E,T)$. From our discussion in section \secref{s:Eisenstein} we know that $E(z)\in \cM_k(4)$ has Fourier coefficients, $\lambda(n)$, satisfying that,  $\lambda(2^{2\nu}m^2)=c_k  2^{2\nu(k-1)}A_k(\nu)\beta_k(m)$ for any odd $m$, where $|A_k(\nu)|\ll 1$ is uniformly bounded. We thus get that 
\begin{eqnarray*}
N_k(E,T)&=&c_k\sum_{\nu\leq \log(T)} 2^{2\nu(k-1)}A_k(\nu) \sum_{m\leq \frac{T}{2^\nu}} \beta_k(m)\\
&=& c_k\sum_{\nu\leq \log(T)} 2^{2\nu(k-1)}A_k(\nu) \cB_k(\frac{T}{2^\nu})
\end{eqnarray*}
Using \propref{p:BkT}, we can estimate
$\cB_k(2^{-\nu}T)=\tilde c_k 2^{(1-2k)\nu}T^{2k-1}+O(2^{(2-2k)\nu}T^{2k-2}),$
to get that
\begin{eqnarray*}
N_k(E,T)&=& c_k\tilde{c}_kT^{2k-1}\sum_{\nu\leq \log(T)} 2^{-\nu}A_k(\nu)+O(T^{2k-2}\log(T))\\
&=& c(E)T^{2k-1}+O(T^{2k-2}\log(T))
\end{eqnarray*}
where $c(E)=c_k\tilde{c}_k\sum_{\nu=1}^\infty \frac{A_k(\nu)}{2^\nu}$.
Since $2k-1=n$ we get that in this case
$$N_k(\Theta_{n+1},T)=  N_k(E,T)+N_k(f,T)= c(E)T^{n}+O(T^{n-1}\log(T))$$
and hence also $N(S^n;T)=c_nT^n+O(T^{n-1}\log(T))$ as claimed. 

Finally we treat the case of $n=2$ separately. Here, instead of relying on modular forms we can use a formula of Hurwitz  (see \cite[Page 3]{Duke2003}) for the series
$$\sum_{q=1}^\infty \frac{r_3(q^2)}{q^s}=6(a-2^{1-s})\frac{\zeta(s)\zeta(s-1)}{L(\omega_{-1},s)},$$
to deduce that 
$$r_3(q^2)=6\sum_{abc=q}\mu(c)\omega_{-1}(c)\chi_0(b)b,$$
where $\chi_0(b)=1$ for $b$ odd and vanishes for even $b$. We thus need to estimate the sum 
\begin{eqnarray*}
N_{3/2}(\Theta_3;T)&=&6\sum_{q\leq T}r_3(q^2)\\
&=& 6\sum_{c\leq T}\mu(c)\omega_{-1}(c)\sum_{ab\leq \frac{T}{c}} \chi_0(b)b.\\
\end{eqnarray*}
We now estimate the inner sum 
\begin{eqnarray*}
S(X)=\sum_{ab\leq X} \chi_0(b)b=S_1(X)+S_2(X)-S_3(X)
\end{eqnarray*}
where as before the first sum is over $a\leq \sqrt{X}$ the second is over $b\leq \sqrt{X}$ and the third is the correction subtracting overcount of $a,b\leq \sqrt{X}$.
First
\begin{eqnarray*}
S_1(X)&=&\sum_{a\leq \sqrt{X}}\sum_{b\leq \frac{X}{a}}\chi_0(b)b\\
&=& \sum_{a\leq \sqrt{X}}(\frac{X^2}{4a^2}+O(\frac{X}{a}))\\
&=& \frac{\zeta(2)}{4}X^2-\frac{X^2}{4}\sum_{a>\sqrt{X}}\frac{1}{a^2}+O(X\log(X))\\
&=&  \frac{\zeta(2)}{4}X^2-\frac{X^{3/2}}{4}+O(X\log(X))\\
\end{eqnarray*}
where we used that  $\sum_{a\geq \sqrt{X}}\frac{1}{a^2}=\frac{1}{\sqrt{X}}+O(\frac{1}{X})$. Next we estimate
\begin{eqnarray*}
S_2(X)
&=& \sum_{b\leq \sqrt{X}}\chi_0(b)b(\frac{X}{b}+O(1))
=  \frac{X^{3/2}}{2}+O(X),
\end{eqnarray*}
and 
\begin{eqnarray*}
S_3(X)&=&\sum_{b\leq \sqrt{X}}\chi_0(b)b\sum_{a\leq \sqrt{X}}1= (\frac{X}{4}+O(\sqrt{X}))(\sqrt{X}+O(1))=\frac{X^{3/2}}{4}+O(X).
\end{eqnarray*}
Combining the three estimates, noting that all terms of order $X^{3/2}$ cancel out we conclude that 
$$S(X)=S_1(X)+S_2(X)-S_3(X)= \frac{\zeta(2)}{4}X^2+O(X\log(X)),$$
and plugging this back we get that 
\begin{eqnarray*}
N_{3/2}(\Theta_3;T)&=& 6\sum_{c\leq T} \mu(c)\omega_{-1}(c)S(\frac{T}{c})\\
&=&  \frac{3\zeta(2)}{2}T^2\sum_{c\leq T}\frac{\mu(c)\omega_{-1}(c)}{c^2}+O(T\log^2(T))\\
&=&  \frac{3\zeta(2)}{2 L(\omega_{-1},2)}T^2+O(T\log^2(T))).\\
\end{eqnarray*}
From this estimate we can use \lemref{l:N2N*} to conclude that 
$$N(S^2;T)=c_2T^2+O(T\log^3(T)),$$
as claimed. 

\subsection{Counting in a subset of $S^n$}\label{s:NOmega}

In order to estimate $N(\Omega;T)=\#\{\frac{\bm{p}}{q}\in \Omega: q\leq T\}$ for a nice set $\Omega\subseteq S^n$, we can express the indicator function of $\Omega$ as a sum of harmonic polynomials on the sphere, reducing the problem to estimating sums of the form 
$$N(P;T)=\mathop{\sum_{\frac{\bm{p}}{q}\in S^n}}_{q\leq T} P(\tfrac{\bm{p}}{q}),$$
with $P(x)$ a homogenous polynomial in $n+1$ variables. We note that if $\nu=\deg(P)$ then the corresponding Theta function 
$$\Theta_{n+1}(P;z)=\sum_{\bm{v}\in \Z^{n+1}} P(\bm{v}) e^{2\pi i \|v\|^2 z}=\sum_{m=1}^\infty r_{n+1}(P;m)e^{2\pi i mz},$$
with $r_{n+1}(P;m)=\sum_{\|\bm{v}\|^2=m}P(\bm{v})$, is a  modular form of weight $k+\nu$ and is a cusp form when $\nu\neq 0$ \cite[Corollary 10.7]{Iwaniec1997book}.
Since we can rewrite 
$$N(P;T)=\sum_{d=1}^\infty \mu(d)\sum_{q\leq \frac{T}{d}}\frac{r_{n+1}(P;q^2)}{q^\nu},$$
 using \propref{p:Ncusp} and summation by parts gives the bound
$$N(P;T)\ll_\nu \| \Theta_{n+1}(P; \cdot)\|_2 T^{\frac{n+1}{2}}\log(T).$$
This estimate together with the same argument given in \cite[Lemma 5.1]{BurrinGrobner2024} gives the bound
$$N(P;T)\ll_\nu \|P|_{S^n}\|_\infty  T^{\frac{n+1}{2}}\log(T).$$
Now, using this bound and following the same argument as in \cite{BurrinGrobner2024} it is possible to get an asymptotic estimate for $N(\Omega;T)$.
However, this bound is not as good as the bound 
$$N(P;T)\ll \| P|_{S^n} \|_\infty  T^{\frac{n}{2}+\epsilon},$$
already given in \cite[Theorem 8.4]{BurrinGrobner2024} using contour integration of $L$-functions.

\subsection{More general ellipsoids}\label{s:SQ}
The method outlined above works for counting rational points on more general ellipsoids $S_Q$ described in Remark \ref{r:SQ}. To do this we note that the Theta function 
$$\Theta_Q(z)=\sum_{\bm{v}\in \Z^{n+1}}e^{2\pi i z Q(\bm{v})}$$
is a modular form in $\Theta_Q\in M_{\frac{n+1}{2}}(N,\chi)$. We can thus use the decomposition of $\Theta_Q=E+f$ as a sum of Eisenstein series and cusp form. The same argument can be used to bound $N_k(f;T)$ and one can use the explicit description of Fourier coefficients of Eisenstein series to estimate $N_k(E;T)$ in the same way (where we need $N=4D$ with $D$ odd and square free when  $k$ is a half integer). The only slight modification is when $n=3$ the space $\cS_2(N,\chi)$ might not be trivial, leading to an extra power of $\log(T)$ for the remainder. When $n=2$, if the space $\cS_{3/2}(N,\chi)$ is non trivial we get an additional term of order $O(T^{3/2})$ so our bound does not extend to general ellipsoids in this case.


\end{document}